\documentclass[12pt]{amsart}
\usepackage{color}
\usepackage{amssymb,mathdots}
\usepackage{soul,xcolor}
\usepackage[a4paper,margin=3cm]{geometry}
\usepackage[all]{xy}
\usepackage[colorlinks,linkcolor=blue,citecolor=blue,urlcolor=blue]{hyperref}

\usepackage[shortlabels]{enumitem}

\usepackage{mathtools}						
\mathtoolsset{showonlyrefs,showmanualtags}	

\newcommand{\g}{\mathfrak{g}}
\newcommand{\m}{\mathfrak{m}}

\newcommand{\Hom}{\operatorname{Hom}}
\newcommand{\Sym}{\operatorname{Sym}}
\newcommand{\ad}{\operatorname{ad}}

\newcommand{\bbR}{\mathbb{R}}

\newcommand{\fs}{\mathfrak{s}}
\newcommand{\fh}{\mathfrak{h}}
\newcommand{\gl}{\mathfrak{gl}}
\newcommand{\fn}{\mathfrak{n}}
\newcommand{\Der}{\operatorname{Der}}
\newcommand{\csp}{\mathfrak{csp}}
\newcommand{\Fib}{\operatorname{Fib}}

\newcommand{\cC}{\mathcal C}
\newcommand{\A}{\mathcal A}
\newcommand{\X}{\mathcal X}
\newcommand{\N}{\mathcal N}
\renewcommand{\P}{\mathbb P}
\newcommand{\PD}{{\P}D}
\newcommand{\PDD}{\PD^\perp\backslash\P(D^2)^\perp}
\newcommand{\mJ}{\mathcal J}

\newcommand{\Id}{\operatorname{Id}}

\usepackage{todonotes}

\setstcolor{red}

\newtheorem{thm}{Theorem}
\newtheorem{cor}{Corollary}
\newtheorem{prop}{Proposition}
\newtheorem{lem}{Lemma}
\theoremstyle{definition}
\newtheorem{defin}{Definition}
\theoremstyle{remark}
\newtheorem{rem}{Remark}

\numberwithin{equation}{section}

\title[Distributions with very large symmetries]{Vector distributions with very large symmetries via rational normal curves}
\author
{Boris Doubrov}
	\address{Belarussian State University\\
	 Nezavisimosti Ave.~4, Minsk 220030\\ Belarus;}
	\email{doubrov@bsu.by}
	
	\author{Igor Zelenko}
	\address{Department of Mathematics\\
		 Texas A$\&$M University\\
		College Station, TX 77843-3368, USA;}
		\email{zelenko@math.tamu.edu}
\urladdr{\url{http://www.math.tamu.edu/~zelenko}}

\thanks{Zelenko was partly supported by NSF grant DMS-1406193 and Simons Foundation Collaboration Grant for Mathematicians 524213.}
\begin{document}
\begin{abstract}
We construct a sequence of rank 3 distributions on $n$-dimensional manifolds for any $n\geq 7$ such that the dimension of their symmetry group grows exponentially in~$n$ (more precisely it is equal to  $\Fib_{n-1}+n+2$, where $\Fib_n$ is the $n$-th Fibonacci number, starting with $\Fib_1=\Fib_2=1$) and such that the maximal order of weighted jet needed to determine these symmetries grows quadratically in ~$n$. These examples are in sharp contrast with the parabolic geometries where the dimension of a symmetry group grows polynomially  with respect to the dimension of the ambient manifold and the corresponding maximal order of weighted jet space is equal to the degree of nonholonomy of the underlying distribution plus~$1$. Our models are closely related to the geometry of certain curves of symplectic flags and of the rational normal curves.
\end{abstract}
	
\maketitle

\section{Introduction}

\subsection{Algebraic statement of the problem: How large can be Tanaka prolongation?}
We start with the pure algebraic setting.
Let $\m$ be a (negatively) graded nilpotent Lie algebra of (negative) depth $\mu$, that is $\m=\oplus_{i=-1}^{-\mu}\m_{i}$. Recall the following 

\begin{defin}
	\label{Tanakadef}
\emph{Tanaka prolongation} of $\m$ is by definition a graded Lie algebra $\g=\g(\m)=\displaystyle{\oplus_{i\in \mathbb Z}} \g_i(\m)$ satisfying the following three properties: 
\begin{enumerate}
	\item  $\g_{-}=\oplus_{i<0}\g_i(\m) = \m$;
	\item (\emph{non-degeneracy}) for any $u\in \g_{i}(\m)$, $i\ge 0$ the condition $[u,\g_{-}]=0$ implies $u=0$;
	\item (\emph{maximality}) $\g$ is the largest graded Lie algebra  satisfying properties (1) and (2). 
\end{enumerate}
\end{defin}

We also assume that $\m$ is generated by $\g_{-1}$. In this case the algebra $\m$ is called \emph{fundamental}.
As described in more detail in subsection \ref{distsec}, a fundamental algebra $\m$ and its Tanaka prolongation $\g$ describe geometrically a left invariant distribution on a Lie group corresponding to $\m$ and its Lie algebra of  infinitesimal symmetries, respectively. Hence, all algebraic questions below can be interpret geometrically as questions on the Lie algebra of infinitesimal symmetries of certain vector distributions.

Denote $\g_{+}=\sum_{i>0}\g_i(\m)$. Note that the degree $0$ component $\g_{0}$ of $\g(\m)$ coincides with the Lie algebra $\Der_0{\m}$ of derivations of $\m$ preserving the grading. In particular, $\g_0\neq 0$, since it contains the grading element, i.e. the elements $e$ such that $[e, x]= i x$ for all $x \in \m_i$, $i<0$.

We say that the Lie algebra $\m$ is \emph{of finite type}, if its Tanaka prolongation is finite-dimensional. According to~\cite{tan}, this is the case if and only if the complexification of $\Der_0(\m)$ contains no non-zero rank 1 elements that annihilate all $\m_{-i}$ for $i\ge 2$. From now on we shall consider that $\m$ satisfies this condition and, thus, $\dim\g(\m) <\infty$. 

We say that the Tanaka prolongation $\g=\g(\m)$ is of positive depth $\nu$, if $\g=\sum_{i=-\mu}^{\nu}\g_i$. The main questions treated in this paper are:
\begin{itemize}
	\item How large can be positive depth of the Tanaka prolongation in comparison with the (negative) depth of $\m$?
	\item How large can be the dimension of the Tanaka prolongation in comparison with the dimension of $\m$. 
\end{itemize}

One class of examples related to these questions comes from graded semisimple Lie algebras. Namely, let $\g$ be a semisimple graded Lie algebra, and let $\m=\g_{-}$. Then except for a few cases (see~\cite{yam}) we have 
\begin{equation}
\label{ss}
\g(\m)=\g.
\end{equation}
In particular, in this class of examples we have $\mu=\nu$ and $\dim\g_{+}=\dim \g_{-}$. Moreover, in all those cases of semisimple $\g$, when  \eqref{ss} does not hold, $\g(\m)$ is infinite dimensional. 

This led to the conjecture that if $\dim \g(\m) <\infty$ the following inequalities are true:
\begin{equation}
\label{mainquest}
\begin{aligned}
&\dim \g_{+} \le \dim \g_{-} (= \dim \m), \\
&\nu \le \mu .
\end{aligned}
\end{equation}
However, we show in this paper that this conjecture does not hold by providing explicit counterexamples.
\subsection{The counterexample}
Let us fix an arbitrary integer $k\ge 2$ and define the fundamental graded nilpotent Lie algebra $\m$ as a $(k+4)$-dimensional vector space with a basis $(X,E_1,\dots,E_k,F_{k-1},F_k, N)$ (the reasons for such notation will become clear later in the paper) and the following non-zero Lie brackets of basis elements:
\begin{gather}
\label{grad1}
\begin{aligned}
	\m_{-1} &= \langle X, E_1, F_{k-1}\rangle;\\
	\m_{-2} &= \langle E_2 = [X,E_1], F_k=[X, F_{k-1}], N = [F_{k-1}, E_1] \rangle;\\
	\m_{-i} &= \langle E_i = [X,E_{i-1}] \rangle,\quad 3\le i \le k.
\end{aligned}
\end{gather}
In particular, its negative depth $\mu$ is equal to $k$. For $k=2$ this Lie algebra corresponds to the parabolic geometry of type $(C_3,\alpha_1)$ (\cite{bryant}). In particular, we have $\mu=\nu=2$ and $\dim\g_{-}=\dim\g_{+} = 6$. 

\begin{thm}
	\label{counterthm}
	For $k\ge 3$ we have 
	\begin{equation}\label{mainid1}
		\dim \g_{\ge 0} = \Fib_{k+3}+2, 
	\end{equation}
	\begin{equation}\label{mainid2}
		\nu = \left\lfloor\cfrac{(k+1)^2}{4}\right\rfloor-2,
	\end{equation}
	where $\Fib_n$ is the $n$-th Fibonacci number (starting with $\Fib_1=\Fib_2=1$).
\end{thm}
Theorem~\ref{counterthm} is proved in sections~\ref{sec2}--\ref{endproof}. 

\textbf{Conclusion.} Theorem~\ref{counterthm}  gives the negative answer to both questions stated in~\eqref{mainquest}. Moreover, we got a sequence of negatively graded nilpotent algebras with dimensions growing to infinity such that the dimensions  of the positively graded part of their Tanaka prolongations grow \emph{exponentially} with respect their dimension and the positive depth of their Tanaka prolongations grows \emph{quadratically} with respect to their negative depth.

\subsection{Geometric consequences: vector distributions with ``very large''  symmetry groups}
\label{distsec}

Negatively graded fundamental Lie algebras  appear as basic local invariants (called Tanaka symbols) of regular vector distributions on manifolds (subbundles of their tangent bundles), while their Tanaka prolongations of a Tanaka symbol are isomorphic to algebras of infinitesimal symmetries of the maximally symmetric distributions with given Tanaka symbol. 

In more detail, given a distribution $D$ on a manifold $M$  one can use the iterative Lie brackets of smooth sections pf $D$ to define a filtration 
\begin{equation}
\label{filtr}
D^{1}(q)\subset D^{2}(q)\subset\ldots
\end{equation}
of the tangent space $T_qM$ for every $q\in M$ by the following recursive formula:  
$$D^{1}(q)=D(q);\quad D^{j}(q):=D^{j-1}(q)+[D,D^{j-1}](q), \quad  j>1.$$
A distribution $D$ is called \emph{bracket-generating} (or \emph{completely nonholonomic}) if for any $q$ there exists $\mu(q)\in\mathbb N$ such that $D^{\mu(q)}(q)=T_q M$.  The minimal positive integer $\mu(q)$ satisfying 
the last identity is called the \emph{degree of nonholonomy} of the distribution $D$ at $q$.

A distribution $D$ is called \emph{regular} at a point $q\in M$ if there is a neighborhood $U$ of $q$ in $M$ such that for each $j<0$ the dimensions of subspaces $D^j(y)$ are constant for all $y\in U$. 

Now assume that $D$ is a bracket generating, $q$ is its regular point, and $\mu$ is its degree of nonholonomy at $q$. 
Set $\m_{-1}(q)\stackrel{\text{def}}{=}D^{1}(q)$ and $\m_{-j}(q)\stackrel{\text{def}}{=}D^{j}(q)/D^{j-1}(q)$ for $1<j<\mu$ and  consider the graded space
\begin{equation}
\label{symbdef}
\mathfrak{m}(q)=\bigoplus_{j=-\mu}^{-1}\m_j(x),
\end{equation}
corresponding to the filtration \eqref{filtr}.

The space $\mathfrak m(q)$  is endowed with the natural structure of a graded Lie algebra. Indeed, take $Y_1\in\m_{-i}(q)$ and $Y_2\in \m_{-j}(q)$ and extend them to the local sections $\widetilde Y_1$ of the distribution $D^i$ and a local section $\widetilde Y_2$ of the distribution $D^j$. Then the bracket $[Y_1,Y_2]\in\m_{i+j}(q)$ is defined as:
\begin{equation}\label{Liebrackets}
[Y_1,Y_2]\stackrel{\text{def}}{=} [\widetilde Y_1,\widetilde Y_2](q) \mod D^{i+j-1}(q).
\end{equation}
It is easy to see that the right-hand side  of \eqref{Liebrackets} does not depend on the choice of local sections $\widetilde Y_1$ and $\widetilde Y_2$. Besides, $\m_{-1}(q)$ generates the whole algebra $\mathfrak{m}(q)$.
The graded nilpotent Lie algebra $\mathfrak{m}(q)$ is called the \emph {Tanaka symbol of the distribution $D$ at the point $q$}.

Now assume that the Tanaka symbols $\mathfrak{m}(q)$ of the distribution $D$ at the point $q$ are isomorphic, as graded Lie algebra, to a fixed fundamental graded Lie algebra $\mathfrak m=\oplus_{i=-\mu}^{-1}\m_i$.
In this case  $D$ is said to be \emph{of constant symbol $\mathfrak{m}$} or just \emph{of type $\mathfrak{m}$}.

If the Tanaka prolongation $\g(\m)$ of $\m$ is finite dimensional, then, as shown by Tanaka \cite{tan,yam}, the algebra $\g(\m)$ of $\m$ is isomorphic to the Lie algebra of infinitesimal symmetries of the so-called \emph{standard (or flat) distribution $D_{\m}$ of constant type $\m$}. To describe the flat distribution, let $M(\m)$ be the simply connected Lie group with the Lie algebra $\m$ and let $e$ be its identity. Then $D_{\m}$ is the left invariant distribution on $M(\m)$ such that $D_{\m}(e)=\m^{-1}$. Besides, the algebra $\g_{\geq 0}=\g_0+\g_+$ is isomorphic to the Lie algebra of the isotropy subgroup of the symmetry group of $D_{\m}$, i.e. of the subgroup of all symmetries preserving the point. Note that the positive depth $\nu$ of $\g(\m)$ is the minimal natural number such that if the infinitesimal symmetry of $D_{\m}$ has the weighted $(\nu+1)$-jet equal to zero then it is trivial or, equivalently, the symmetries are determined by $(\nu+1)$-st weighted jet. 

With all these constructions and terminology our Theorem \ref{counterthm} and Conclusion of the previous subsection implies the following:
\begin{cor}
	\label{geominterp}
	For any $k\ge 3$ there exists a rank $3$ distribution on a $(k+4)$-dimensional manifold with degree of nonholonomy $k$ such that the dimension of its Lie algebra of infinitesmal symmetries is equal to $\Fib_{k+3}+k+6$  and its local symmetries are determined by $\left(\left\lfloor(k+1)^2/4\right\rfloor-1\right)$-st weighted jet.
\end{cor}

The geometric constructions behind our models are briefly described in section~\ref{geomsec} (for more details see~\cite{DZ2016}).

\section{Partial Tanaka prolongations}
\label{sec2}

By \emph{a partial Tanaka prolongation} we mean any graded Lie algebra that satisfies the first two properties in Definition~\ref{Tanakadef} of the Tanaka prolongation, but not necessarily the third condition (of maximality).

\subsection{First partial prolongation}
Define \emph{the first partial prolongation $\g'$} of $\m$ as follows:
\begin{gather}\label{firstP}
\begin{aligned}
	&\g'= \langle E_0, E_1, \dots, E_k\rangle + \langle F_0, F_1, \dots, F_k\rangle + \langle N \rangle + \langle X \rangle;\\
	& [E_i,F_{k-i}] =(-1)^i N, \quad i=0,\dots,k;\\ 
	&[X,E_i]=E_{i+1}, \quad [X,F_i]=F_{i+1},\quad i=0,\dots, k-1. 
\end{aligned}
\end{gather}

The degrees of additional basis elements are:
\begin{gather}
\label{grad2}
\begin{aligned}
&\deg E_0 = 0; \\
&\deg F_i = k-2-i,\quad i=0,\ldots,k-2.
\end{aligned}
\end{gather}
It is easy to see that with these definitions $\g'$ is a graded Lie algebra and for any non-zero element $u\in\g'_{\ge0}$ we have $[X,u]\ne 0$.

\subsection{Second partial prolongation}
Further, let 
\begin{eqnarray}
&\fn=\langle E_0,\dots,E_k,F_0,\dots,F_k, N\rangle; \label{Heisen}\\
&\fn_{-1}=\langle E_0,\dots,E_k,F_0,\dots,F_k\rangle , \quad \fn_{-2}=\langle N\rangle; \label{symplect}
\end{eqnarray}
Then $\fn$ is the $(2k+3)$-dimensional Heisenberg algebra and the subspace $\fn_{-1}$ is endowed with the line of symplectic forms (\emph{the canonical conformal symplectic structure on $\fn_{-1}$}), generated, for example, by the symplectic form $\sigma$ given by 
\begin{equation}
\label{sigma}
[X, Y]=\sigma(X, Y)\, N, \quad \forall X, Y \in \fn_{-1}.
\end{equation}

Hence the Lie algebra $\Der_0(\fn)$ (with respect to the grading $\fn=\fn_{-1}\oplus \fn_{-2}$) can be identified with the subalgebra $\csp(\fn_{-1})$ of $\gl(\fn_{-1})$ consisting of all endomorphisms preserving the canonical conformal symplectic structure on $\fn_{-1}$. 

Next, define
\begin{equation}
E=\langle E_0,\dots,E_k\rangle , \quad F=\langle F_0,\dots,F_k\rangle.\label{EF} 
\end{equation}
Note that by \eqref{firstP} and \eqref{symbdef}, the subspaces $E$ and $F$ are Lagrangian with respect to the symplectic form $\sigma$. In particular, $\sigma$ defines the non-degenerate pairing of the spaces $E$ and $F$, and thus, induces the identification
\begin{equation}
\label{EFid}
E\stackrel{\sigma}{\cong} F^*.
\end{equation}
As the form $\sigma$ is defined up to a multiplication by a nonzero constant, this identification is also defined up to a multiplication by a nonzero constant.

Fixing the basis $(E_0,\dots,E_k,F_0,\dots,F_k)$ in $\fn_{-1}$  we have the following identification
\begin{equation}
\csp(\fn_{-1}) \cong  \left\{ \begin{pmatrix} A+\lambda & C \\ B & -A'+\lambda \end{pmatrix}\,: A\in\gl(k+1,\bbR), B=B', C=C' \right\}.
\end{equation}
Here for any $(k+1)\times(k+1)$ matrix $X$ we denote by $X'$ the matrix $J_{k+1}X^tJ_{k+1}^{-1}$, where $J$ is $(k+1)\times (k+1)$ matrix with the $(i, j)$-th entry equal to $\sigma (E_i, F_j)=(-1)^i\delta_{i,k-j}$, $0\le i,j\le k$.

Let $\cC$ be a rational normal curve in the projectivization of the space $E$ such that in the basis $(E_0,\dots,E_k)$ it is given by 
\begin{equation}
\label{rnc}
\cC=\left[u^k:u^{k-1}v: \cfrac{u^{k-2}v^2}{2!}: \cdots:\cfrac{v^k}{k!}\right].
\end{equation}
Note that the operator $\ad X$, which acts on $E$ by $\ad X (E_i) = E_{i+1}$ for $i=0,\dots,k-1$, is an infinitesimal symmetry of $\cC$, that is the one-parametric group $\{\exp(\ad tX): t\in\mathbb R\}$ preserves $\cC$ for all $t\in\bbR$. 

It is well-known (see, for example~\cite{FH1991}) that the symmetry group of $\cC$, is isomorphic to $PGL(2,\bbR)$ irreducibly embedded into $PGL(k+1,\bbR)$. Denote by $\gl(2,\bbR)$ the corresponding subalgebra of $\gl(E)=\gl(k+1,\bbR)$. 

Denote by $I_2(\cC)\subset\Sym^2 E^*$ the space of all homogeneous polynomials of degree $2$ on $E$ vanishing on $\cC$. By constructions, $I_2(\cC)\subset \Sym^2 E^*$ is a $\gl(2, \bbR)$-submodule of $\Sym^2 E^*$. It is well known  $I_2(\cC)$ is exactly the sum of all irreducible $\gl(2, \bbR)$-submodules, except the top dimensional one, in the decomposition of $\Sym^2 E^*$ into the sum of all irreducible $\gl(2, \bbR)$-submodules  (see~\cite[\S11.3]{FH1991}). In other words, if $V_k$ denotes the $(k+1)$-dimensional irreducible  $\gl(2, \bbR)$-submodule, then 
\begin{equation}
\label{sym2decomp}
\Sym^2 E^*=\bigoplus_{i=0}^{\lfloor k/2\rfloor} V_{2k-4i}\end{equation}
and 
\begin{equation}
\label{ICdecomp}
I_2(\cC)=\bigoplus_{i=1}^{\lfloor k/2\rfloor} V_{2k-4i}.\end{equation}
Note that by subspaces $V_{2k-4i}$ in equations \eqref{sym2decomp} and \eqref{ICdecomp} we mean the embeddings of the corresponding abstract $\gl(2, \bbR)$-modules into the $\gl(2, \bbR)$-module $\Sym^2 E^*$.

Using~\eqref{EFid} we can identify $\Sym^2 E^*$ with the corresponding self-adjoint linear maps from $E$ to $E^*\cong F$.
Then, using the grading on $\Hom(E, F)\cong E^*\otimes F$ inherited from the grading on $\fn\subset \g'$,  we get that the highest weight vector (i.e. the vector annihilated by the action of $X$ on  $E^*\otimes F$) of the $\gl(2,\bbR)$-module $V_{2k-4i}$ in the decomposition \eqref{sym2decomp} has degree $2i-2$. Since $\deg X = -1$, this is the minimal degree of elements in $V_{2k-4i}$. This and \eqref{ICdecomp} gives the following characterization of $I_2(\cC)$:
\begin{lem}
	\label{lem+}
$I_2(\cC)$ coincides with the subspace of $\Sym^2 E^*\subset \Hom(E, F)$ spanned by all $\gl(2,\bbR)$-submodules concentrated in non-negative degree with respect to the grading on $\Hom(E, F)$ inherited from the grading of $\fn\subset\g'$.
\end{lem}

Now define \emph{the second partial prolongation $\g''$} as follows:
\begin{equation}
\g''= \fs \ltimes \fn, 
\end{equation}
where $\fs$ is the following subalgebra in $\csp(2k+2,\bbR)$:

\begin{equation}
\label{sdef}
	\left\{ \fs = \begin{pmatrix} A+\lambda\Id & 0 \\ B & -A'+\lambda\Id\end{pmatrix}\,: A\in \gl(2,\bbR), B\in I_2(\cC), \lambda\in\bbR\right\},
\end{equation}
or shortly
\begin{equation}
\label{sdecomp}
\fs=\gl(2,\mathbb R)\oplus \langle \Id \rangle \oplus I_2(\cC).
\end{equation}

Finally, note that by Lemma \ref{lem+} we added only elements in non-negative degree and that none of them commutes with $\g''_{-}=\m$.  So, $\g''$ is indeed a partial Tanaka prolongation of $\m$.

\subsection{Second partial prolongation as symmetry algebra of a curve of symplectic flags}

Note that $\fn$ is equipped with two different gradings: the first one is inherited from the embedding $\fn\subset \g'$ and the second one is the standard grading of Heisenberg Lie algebra, that is~$\fn=\fn_{-1}\oplus \fn_{-2}$ with the notation as in \eqref{symplect}.

\begin{prop}
	\label{prop1}
	 Equip $\csp(\fn_{-1})$ with the first grading (i.e., the grading inherited from the grading of $\fn\subset \g'$). Assume that $k\ge 3$. Then the subalgebra $\fs$ is the largest graded subalgebra $\fs$ in $\csp(\fn_{-1})$ such that $\fs_{-} = \langle \ad X \rangle$.
\end{prop}
\begin{proof}
	Let $\fs'$ be the largest among all graded subalgebras $\fs$ in $\csp(\fn_{-1})$ such that $\fs_{-} = \langle \ad X \rangle$. Since the subalgebra $\gl(2,\bbR)$ does satisfy this condition, we have $\gl(2,\bbR)\subset \fs'$. 
	
	Consider now the action of $\gl(2,\bbR)$ on $\csp(\fn_{-1})$. The algebra $\csp(\fn_{-1})$ is decomposed as an $\gl(2,\bbR)$-module into the direct sum of irreducible $\gl(2,\bbR)$-submodules. Since $\fs_{-}'= \langle \ad X \rangle$, it follows that, apart from $\gl(2,\bbR)$ itself, all other $\gl(2,\bbR)$ submodules inside $\fs'$ must be concentrated in non-negative degree.
	
	It is easy to see that the converse is also true. Indeed, the bracket of any such two submodules is also concentrated in the non-negative degree, and hence the sum of all such submodules forms a subalgebra. Note that 
		\begin{equation}
		\label{decomp}
		\csp(\fn_{-1})\cong \gl(E)\oplus \langle \Id\rangle \oplus \Sym^2 E\oplus \Sym^2 F,
		\end{equation}
where, using the identification \eqref{EFid}, $\Sym^2 F$ is embedded into $\Hom(E, F)$ and $\Sym^2 E$ is embedded into $\Hom(F, E)$. Then it is easy to see that for $k\ge 3$ the component $\Sym^2 E$ in the decomposition~\eqref{decomp} does not contain any irreducible $\gl(2,\bbR)$-submodules concentrated in non-negative degree, the component~$\gl(E)$ contains no such submodules complementary to $\gl(2,\bbR)$, and for $\Sym^2 F$ ($\cong \Sym^2 E^*$)  we can use Lemma \ref{lem+}. 
Hence $\fs'=\gl(2,\bbR)+\langle \Id\rangle + I_2(\cC)=\fs$.  
\end{proof}

\begin{rem}
\label{flatcurve}
Using the first grading (i.e., the grading inherited from the grading of $\fn\subset \g'$), we can define the decreasing  filtration $\{J^i\}_{i=-k}^{k-1}$ on the space $\fn_{-1}$ from \eqref{Heisen}, where $J^i$ is spanned by all elements in $\fn_{-1}$ of first degree $\ge i$. Note that $J^i$ is the skew-orthogonal complement (with respect to the symplectic form $\sigma)$ of $J^{-i-1}$. In particular, the skew-orthogonal complements of $\{J^i\}_{i=-k}^{k-1}$ form the same flag. A flag satisfying the latter property is called~\emph{symplectic}. Let $\mathcal F_X$ be the orbit of the flag $\{J^i\}_{i=-k}^{k-1}$ under the action of the one-parameter subgroup $\{\exp\ad tX\}\subset \mathrm{CSp}(\fn_{-1})$. Obviously, this is a curve of symplectic flags. It turns out (see~\cite{DZ2013}) that the algebra $\fs$ is isomorphic to the subalgebra in $\csp(\fn_{-1})$ consisting of all infinitesimal symmetries of the curve $\mathcal F_X$. 
\end{rem}

\section{Tanaka prolongation for the second grading via secant varieties of rational normal curves}
\label{secantsec}
Recall that the Tanaka prolongation can be defined not only for negatively graded Lie algebras, but also for non-positively graded Lie algebras. The only change that we have to make in Definition \ref{Tanakadef} in the latter case is that in property~(1) we require that the non-positively graded part of the Tanaka prolongation coincides with the original non-positively graded Lie algebra.

Now consider the Heisenberg algebra  $\fn$ with the standard grading  $\fn=\fn_{-1}\oplus \fn_{-2}$ with the notation as in \eqref{symplect} and the algebra $\fs$ as a subalgebra in $\Der_0(\fn)\cong \csp(\fn_{-1})$. Then $\fn\oplus \fs$ it is the non-positively graded Lie algebra (with $\fs$ being the component with degree zero. Denote the Tanaka prolongation of $\fn\oplus\fs$ by $\g(\fn, \fs) = \displaystyle{\oplus_{i\in \mathbb Z} \g_i(\fn, \fs)}$. Now we calculate $\g(\fn, \fs)$, as it is crucial for the proof of our main  Theorem \ref{counterthm} (see Proposition \ref{propequal} in section \ref{endproof}).


 As above, let $\cC$ be the rational normal curve in $P^k = P(E)$. Recall that the $r$-th secant variety $\sigma_r\cC$ is the Zariski closure of the union of all linear spaces in $P^k$ spanned by collections of $r+1$ points on $\cC$. In particular, $\sigma_0\cC = \cC$ and $\sigma_1\cC$ is the closure of the union of all lines passing through two points on $\cC$. 
 
 Similar to defining $\cC$ by quadratic polynomials from $I_2(\cC)$, it is well-known~\cite{IK1999,LO2013} that the secant varieties $\sigma_r\cC$ can be defined as zeros of all polynomial of degree $(r+2)$ from $I_{r+2}(\sigma_r\cC)$. The space $I_{r+2}(\sigma_r\cC)$, in its turn, can be explicitly described by the space of (linearly independent) maximal minors of the catalecticant (or Hankel) matrix:
 \begin{equation}\label{Hankel}
 \begin{pmatrix}
 0!\,x_0 & 1!\,x_1 & 2!\,x_2 & \dots & (k-r-1)!\, x_{k-r-1} \\
 1!\,x_1 &2!\, x_2 & 3!\,x_3 & \dots & (k-r)!\, x_{k-r} \\
 \vdots & & & & \vdots \\
 (r+1)!\,x_{r+1} & (r+2)!\,x_{r+2} & (r+3)!\, x_{r+3} & \dots & k!\, x_k
 \end{pmatrix},
 \end{equation}
 where $[x_0:x_1:\dots:x_k]$ are homogeneous coordinates in $P^k$ with respect to the basis $(E_0, \ldots , E_k)$.
 
In order to compute $\g(\fn,\fs)$ consider first the Tanaka prolongation $\g(\fn)$ of the Heisenberg Lie algebra $\fn$. 
It is well known that it can be identified with the Lie algebra of polynomial contact vector field on $\bbR^{2k+3}$, where the contact structure is given as the annihilator of the following 1-form:
\begin{equation}\label{contactform}
	\theta = dz + \cfrac{1}{2}\sum_{i=0}^k \big(x_i\,dy_i - y_i\, dx_i\big). 
\end{equation}

All such contact vector fields $X_f$ are parametrized by a single polynomial function $f(x_0,\dots,x_k,y_0,\dots,y_k,z)$ such that
\begin{equation}
(L_{X_f} \theta)\wedge \theta = 0, \quad 	\theta(X_f)=f.\
\end{equation} 
Here the first relation is means that the flow of $X_f$ preserves the contact distribution and the second relation is the normalization condition.

By direct computations, in the chosen coordinates
	\begin{equation}
	\label{contcoord}
	X_f=\sum_{i=0}^k\left( -\cfrac{df}{dy_i}\cfrac{d}{dx_i}+\cfrac{df}{dx_i}\cfrac{d}{dy_i}\right)+ f\cfrac{\partial}{\partial z},
	\end{equation}
where 
\begin{equation}
\label{dd}
\frac{d}{dx_i} = \frac{\partial}{\partial x_i} + \frac{y_i}{2} \frac{\partial}{\partial z}, \quad \frac{d}{dy_i} = \frac{\partial}{\partial y_i} - \frac{x_i}{2}\frac{\partial}{\partial z}, \quad i=0,\dots, k.
\end{equation}
The Lie bracket of contact vector fields induces the bracket $\{f,g\}$ of polynomials~$f,g$:
\begin{equation}
	X_{\{f,g\}} = [X_f,X_g].
\end{equation}
Explicitly, we have:

\begin{equation}
\label{bracketdef}
\{f,g\} = f\frac{\partial g}{\partial z} - g\frac{\partial f}{\partial z} +
\sum_{i=0}^k \left(\frac{df}{dx_i}\frac{dg}{dy_i} - \frac{dg}{dx_i}\frac{d f}{dy_i} \right).
\end{equation}

Now we identify $\g(\fn)$ with the space of polynomial functions $f$ equipped with the above Lie bracket. 
First, the Lie algebra $\fn$ itself can be represented in the following form:
\begin{equation}
\begin{aligned}
\label{EFxy}	
\fn_{-2} &= \langle 1 \rangle,\quad \fn_{-1} = E\oplus F,\text{ where}\\
E &= \langle y_0, y_1, \dots, y_n \rangle, \quad  E_i=y_i,\\
F &= \langle x_0, x_1, \dots, x_n \rangle \cong E^*, \quad F_i=(-1)^i x_{k-i}.
\end{aligned}
\end{equation}
Using standard formulas for the irreducible embedding of $\gl(2, \mathbb R)$ into  $\gl(k+1, \mathbb R)$ (see~\cite[\S11.1]{FH1991}) and \eqref{contcoord}, we get that the subalgebra $\fs\subset \g_0(\fn)$ is represented as follows:
\begin{equation}
\label{scoord}
\begin{aligned}
\fs = \langle X &=x_0y_1 + x_1y_2 + \dots + x_{k-1}y_k, \\
H &=\sum_{i=0}^k (k-2i)x_iy_i, \\
Y &=\sum_{i=1}^k i (k+1-i) x_iy_{i-1},\\ 
Z_1&=x_0y_0+x_1y_1+\dots+ x_ky_k\rangle \oplus\langle Z_2=z\rangle \oplus I_2(\cC),
\end{aligned}
\end{equation}
where $I_2(\cC)$ is represented by $2\times 2$ minors of the matrix
\begin{equation}
\begin{pmatrix}
\label{Hankel2}
0!\, x_0 & 1!\,x_1 & \dots & (k-1)!\,x_{k-1} \\
1!\,x_1 & 2!\, x_2 & \dots & k!\, x_k
\end{pmatrix}.
\end{equation}
 
 \begin{thm}\label{gns}
 	Assume $k\ge 3$. Then
 	\begin{equation}
 	\g(\fn,\fs) = \g'' + \sum_{r=1}^{[k/2]-1} I_{r+2}(\sigma_r\cC).
 	\end{equation}
 \end{thm}

 \begin{proof}
First, recall that given vector spaces $V$ and $W$, and an arbitrary subspace $L$ of $\Hom(V,W)\cong V^*\otimes W$, the \emph {(standard) first prolongation $L^{(1)}$ of $L$} is a subspace in $\Hom(V, L)= V^*\otimes L$ such that:
\begin{equation}\label{standard1}
L^{(1)}=\{\phi\in \Hom(V, L):\phi(v_1)v_2 - \phi(v_2)v_1 = 0,\,\,\text{for all }v_1,v_2\in V\}
\end{equation}
or, shortly,
\begin{equation}
\label{standard1'}
L^{(1)}= \left(V^*\otimes L\right)\cap \left(\Sym^2 V^*\otimes W\right).
\end{equation}
Here we treat both spaces in the right hand side are subspaces of $V^*\otimes V^*\otimes W$. Then one can define the \emph{standard $i$-th prolongation} $L^{(i)}$ of $L$ recursively by 
\begin{equation}
\label{recstand}
L^{(i)}:=\left(L^{(i-1)}\right)^{(1)}, 
\end{equation}
which also equivalent to:
\begin{equation}
\label{standardk}
L^{(i)}= \left(\left(V^*\right)^{\otimes i}\right) \otimes L\cap\left( \Sym^{i+1} V^*\otimes W\right).
\end{equation}
\begin{rem}\label{derprolong}
By \eqref{standardk} the elements of $L^{(i)}$ can be identified with $W$-valued degree $i+1$ homogeneous polynomials on $V$ such that all their partial derivatives of order $k$ (in the  coordinate system on $V$ with respect to any basis) belong to $L$ (see~\cite{kob,stern}).
\end{rem}
   	
\begin{lem}\label{lem1}
The standard $r$-th  prolongation  $I_2(\cC)^{(r)}$ of $I_2(\cC)$, considered as the subspaces of $\Hom(E, F)\cong E\otimes E^*$, coincides with  $I_{r+2}(\sigma_r\cC)$.
\end{lem}
\begin{proof}
From the fact that $I_2(\cC)\subset \mathrm{Sym}^2(E^*)$ and Remark~\ref{derprolong} it follows immediately that $I_2(\cC)^{(r)}$ can be identified with  the space of polynomials of degree $r+2$ in variables $(x_0,\dots,x_k$) such that all their $r$-th partial derivatives belong to $I_2(\cC)$. Let us show that a polynomial $F$ satisfies this property, if and only if $F\in I_{r+2}(\sigma_r\cC)$.
	
Indeed, let $F$ be such polynomial and let $p_0,\dots,p_r$ be an arbitrary set of $r+1$ points on $\cC$. Consider the linear map defining the embedding of the corresponding secant space through the points $p_0,\dots,p_r$:
	\begin{equation}
	W\colon P^r\to P^k, \quad [z_0:z_1:\cdots:z_r]\mapsto
	z_0p_0+z_1p_1+\dots z_rp_r.
	\end{equation}
	Then $F\circ W$ is a polynomial of degree $r+2$ on $P^r$ that vanishes at $(r+1)$ points 
	\begin{equation}
	[1:0:\dots:0],[0:1:\dots:0],\dots, [0:0:\dots:1]\in P^r,
	\end{equation} 
	and, in addition, all its derivatives of degree $\le r$ also satisfy this property. This immediately implies that $F\circ W=0$. Hence, $F$ vanishes
	identically at $\sigma_r\cC$, and thus belongs to $I_{r+2}(\sigma_r\cC)$. Conversely, it is easy to see that if $F\in I_{r+2}(\sigma_r\cC)$ then all its first partial derivatives lie in
	$I_{r+1}(\sigma_{r-1}\cC)$. By induction, this implies that all $r$-th partial derivatives of $F$ belong to $I_2(\cC)$.
\end{proof}
	
By Lemma \ref{lem1}, to prove Theorem \ref{gns} it is sufficient to prove that 
\begin{equation}
\label{goalfin}
\g_i(\fn, \fs)=I_2(\cC)^{(i)}, \quad \forall i\geq 1
\end{equation}
We first prove \eqref{goalfin} for $i=1$, then for $i=2$, and then by induction for $i\geq 3$

{\bf 1.} \emph{ Proof of the identity \eqref{goalfin} for $i=1$.}

First, directly from definition of the Tanaka prolongation $\g_1(\fn, \fs)$ can be identified with the space of all degree $1$ maps $\phi\colon\fn\to\fn\oplus\fs$
such that
\begin{equation}
\label{mod1}
\phi(v_1)v_2 - \phi(v_2)v_1 = \phi([v_1, v_2]) ,\,\,\text{for all }v_1,v_2\in \fn.
\end{equation}
Taking into account \eqref{sigma}, the last relation yields 
\begin{equation}
\label{mod1'}
\phi(v_1)v_2 - \phi(v_2)v_1 = \sigma(v_1, v_2) \phi(N),\,\,\text{for all }v_1,v_2\in \fn_{-1}
\end{equation}
where, as before, $\sigma$ represents the canonical symplectic structure on $\fn_{-1}$ and $N$ generates $\fn_{-2}$.

Further, $I_2(\cC)^{(1)}$ can be naturally embedded into $\g_1(\fn, \fs)$. Indeed if $\phi\in I_2(\cC)^{(1)}\subset \mathrm{Hom}(E, F)$, then by extending it to be equal to zero to $F$ and $\langle N\rangle$ and using \eqref{standard1} we obtain an element of $\g_1(\fn,\fs)$, which defines the required embedding.
  	
To prove the opposite inclusion it is sufficient to prove that for an arbitrary  $\phi\in \g_1(\fn, \fs)$, one has $\phi(N)=0$ and
$\phi|_{\fn_{-1}}$ takes values in~$I_2(\cC)$. 

Now analyze \eqref{mod1'} separately in the following cases:

\textbf{Case 1:} $v_1,v_2\in F$. Since the subspace $F$ is Lagrangian with respect to $\sigma$ we get $\phi(v_1)v_2=\phi(v_2)v_1$, which is exactly 
the equation for the standard prolongation of the restriction of $\fs$ to
$F$. This restriction is equal to the image of the irreducible embedding of
$\gl(2,\mathbb R)$ into $\gl(F)$, because  the restriction of any element of $I_2(\cC)\subset\gl(\fn_{-1})$ to $F$ is zero.

According to the result of Kobayashi--Nagano~\cite{kobnag}, the first prolongation of the irreducible embedding of $\gl(2,\mathbb R)$ is non-zero only if the dimension of the representation space does not exceed~$3$. In our case the lowest possible dimension of $F$ is equal to $4$. Thus, we see that $\phi(v_1)v_2=0$ for all $v_1,v_2\in F$. In particular, this implies that $\phi(F)$ lies in $\mathbb R (Z_1-Z_2)+I_2(\cC)$ (here $Z_1$ and $Z_2$ are as in \eqref{scoord}), because the latter is exactly the subspace of all elements in $\fs$  for which the restriction to $F$ is equal to $0$.

\textbf{Case 2:} $v_1,v_2\in E$. As $E$ is also Lagrangian, we also get $\phi(v_1)v_2=\phi(v_2)v_1$. Considering this equation modulo $F$, we again get that $\phi(v_1)v_2 = 0 \mod F$. This implies that $\phi(E)\subset \mathbb R(Z_1+Z_2)+I_2(\cC)$.

Cases 1 and 2 above imply that $\phi$ can be decomposed as follows:
\begin{equation}\label{phipr}
\begin{split}
&\phi(f+e)=\bar\phi(f+ve) + \alpha(f)(Z_1-Z_2)/2 +
\beta(e)(Z_1+Z_2)/2,\\
&\forall f\in F, e\in E,
\end{split}
\end{equation}
where $\bar\phi$ takes values in $I_2(\cC)$, $\alpha\in F^*$ and $\beta\in E^*$.

\textbf{Case 3:} $v_1= f\in F$, $v_2=e\in E$. We have:
\[
\phi(f)e - \phi(e)f = \sigma(f,e) \phi(N).
\]
Considering this equation modulo $F$ and taking into account that
$I_2(\cC)$ vanishes on $F$ and sends $E$ to $F$, we get:
\[
\alpha(f)e=\sigma(f,e)N \mod \,F, \quad\text{for any }f\in F,
e\in E.
\]
As the dimensions of $E$ and $F$ are at least 4, for any $f$ we can find
a non-zero vector $e\in E$ such that $\sigma(f,e)=0$. Hence, we see
that $\alpha(f)=0$. In particular, this implies that $\phi(N)\in F$.

Let us now fix an arbitrary $f\in F$. From~\eqref{phipr} we have
$\phi(f)=\bar\phi(f)$. Let us prove that $\bar\phi(f)=0$. Indeed, we have
\[
\bar\phi(f)e = \sigma(f,e)\phi(N) + \beta(e)f, \quad\text{for any }e\in
E.
\]
The element $\bar\phi(f)$ lies in $I_2(\cC)$, while the right-hand side defines a certain linear map from $E$ to $F$, which has rank $\le 2$. However, from the description of $I_2(\cC)$ as maximal minors of the Hankel  matrix \eqref{Hankel2} it follows that  the space $I_2(\cC)$ does not contain any non-zero elements of rank $2$ or less. Thus, we get $\bar\phi(f)=0$ and $\phi(N)$ is proportional to $f$. As $f$ can be arbitrary and the dimension
of the subspace $F$ is at least $4$, this implies that $\phi(N)=0$ and $\beta=0$.
		
Hence, $\phi$ takes values in $I_2(\cC)$ and, thus, $\g_1(\fn,\fs)$ coincides with $I_2(\cC)^{(1)}$, as desired.

\medskip
{\bf 2.}\emph{ Proof of the identity \eqref{goalfin} for $i=2$.}
Here we again consider $\g(\fn, \fs)$ as a subalgebra of the algebra of contact vector fields or the algebra of polynomials with respect to the brackets given by \eqref{bracketdef}. Denote $\mathbf{x}=(x_0, \ldots, x_k)$, $\mathbf{y}=(y_0, \ldots, y_k)$. In this description $\g_i(\fn, \fs)$ can be recursively identified with the space of homogeneous polynomials $\g(\mathbf{x},\mathbf{y},z)$ of weighted degree $i+2\ge 4$, where the weight of variables $x_i$ and  $y_i$  is $1$ and the weight of $z$ is $2$, such that the polynomials $\{x_i, g\}$ and $\{y_i, g\}$ belong to $\g_{i-1}(\fn, \fs)$ for all $i=0, \ldots k$.

For $i=2$ a homogeneous polynomial $g$  of weighted degree $4$ has the form
\begin{equation}
\label{deg4}
g=Q_0(\mathbf{x}, \mathbf{y})+Q_1(\mathbf{x},\mathbf{y})z+Q_2z^2,
\end{equation}
where $Q_2$ is constant, $Q_1(\mathbf{x}, \mathbf{y})$ is a quadratic polynomial, and $Q_0(\mathbf{x}, \mathbf{y})$ is homogeneous polynomial of degree $4$ and by part 1 we have $g\in \g_2(\fn, \fs)$ if and only if
\begin{equation}
\label{brackcond1}
\{x_j, g\}\in I_2(\cC)^{(1)}, \quad \{y_j, g\}\in I_2(\cC)^{(1)}, \quad \forall j=0, \ldots k
\end{equation}
From \eqref{bracketdef} and \eqref{dd} it is easy to show that 
\begin{equation}
\label{brackxy}
\{x_j, g\}=\cfrac{\partial g}{\partial y_j}+\cfrac{x_j}{2}\cfrac{\partial g}{\partial z} , \quad \{y_j, g\}=-\cfrac{\partial g}{\partial x_j}+\cfrac{y_j}{2}\cfrac{\partial g}{\partial z}, \quad \forall j=0, \ldots k.
\end{equation}
Applying \eqref{brackxy} to $g$ as in \eqref{deg4} we get:
\begin{equation}
\begin{aligned}
\{x_j, g\}=\cfrac{\partial Q_0}{\partial y_j}+\cfrac{\partial Q_1}{\partial y_j}z+\cfrac{x_j}{2}\left(Q_1+2 Q_2z\right) =\\ \left(\cfrac{\partial Q_0}{\partial y_j} +\cfrac{x_j}{2}Q_1 \right)+\left(\cfrac{\partial Q_1}{\partial y_j}+Q_2 x_j\right)z, \label{brackx}
\end{aligned}
\end{equation}
\begin{equation}
\begin{aligned}
\{y_j, g\}=-\cfrac{\partial Q_0}{\partial x_j}-\cfrac{\partial Q_1}{\partial x_j}z+\cfrac{y_j}{2}\left(Q_1+2 Q_2z\right) =\\ \left(-\cfrac{\partial Q_0}{\partial x_j} +\cfrac{y_j}{2}Q_1 \right)+\left(-\cfrac{\partial Q_1}{\partial x_j}+Q_2 x_j\right)z. \label{bracky}
\end{aligned}
\end{equation}
Recall that all elements of $I_2(\cC)^{(1)}$ are polynomials in $\mathbf{x}$. Therefore from \eqref{brackcond1} it follows that the polynomial factors  near $z$ in the last two equations must vanish, i.e., 
\begin{equation}
\label{r1} 
\left\{\begin{array}{l}
\cfrac{\partial Q_1}{\partial y_j}+Q_2 x_j=0\\
-\cfrac{\partial Q_1}{\partial x_j}+Q_2 y_j=0
\end{array}\right.
\end{equation}
Differentiating the first equation in \eqref{r1}  with respect to $x_j$ and the second equation with respect to  $y_j$, summing up the results,  and using that $Q_2$ is constant, we get that $Q_2=0$.
This in turn implies that   $\cfrac{\partial Q_1}{\partial x_j}=\cfrac{\partial Q_1}{\partial y_j}=0$ for every  $j=0, \ldots k$. Since $Q_1$ is a quadratic polynomial, we get that $Q_1=0$.

Substituting $Q_1=Q_2=0$ into \eqref{brackx} and \eqref{brackx} -\eqref{bracky}
\begin{equation}	
\label{inprol2}
\cfrac{\partial Q_0}{\partial x_j}, \cfrac{\partial Q_0}{\partial y_j} \in I_2(\cC)^{(1)}\quad\text{for any }j=0,\dots,k.
\end{equation}

 Hence, first, $Q_0$ is independent of $\mathbf{y}$ and second, using Remark \ref{derprolong}, we get that $g=Q_0(\mathbf{x})$ belongs to $I_2(\cC)^{(2)}$. So, we proved that $\g_2(\fn, \fs)\subset  I_2(\cC)^{(2)}$. The opposite inclusion is obvious, so we proved \eqref{goalfin} for $i=2$.
\medskip

{\bf 3.}\emph{ Proof of the identity \eqref{goalfin} for $i\geq 3$.}
We make the proof by induction. Assume  that \eqref{goalfin} holds for all $1\leq r<i$  and prove it for $i$. If $g\in \g_{i}(\fn,\fs)$, then 
\begin{equation}
\label{degi+2}
g=\sum_{r=0}^{[\frac{i}{2}]+1}Q_r(\mathbf{x}, \mathbf{y})z^r,
\end{equation}
where $Q_r(\mathbf{x}, \mathbf{y})$ are homogeneous polynomials of degree $i+2-2r$. Further, note that 
the constant polynomial $1$ has weight $-2$ in algebra $\g(\fn)$. Therefore using \eqref{bracketdef}, the grading condition in the algebra $\g(\fn)$, and he induction hypothesis we get that
\[
	\{1, g\}=\frac{\partial g}{\partial z} \in \g_{i-2}(\fn, \fs)=I_2(\cC)^{(i-2)}
\]
In particular,  $\frac{\partial g}{\partial z}$ depends on $x$ only. Consequently $g$ has the form
\begin{equation}
\label{degi+2s}
g=Q_0(\mathbf{x}, \mathbf{y})+Q_1(\mathbf{x})z
\end{equation}
From this and \eqref{brackxy} it follows that for every $j=0, \ldots, k$
\begin{eqnarray}
&~&\{x_j, g\}=\cfrac{\partial Q_0}{\partial y_j}+\cfrac{x_j}{2}Q_1,  \label{brackxi} \\
&~&\{y_j, g\}=-\cfrac{\partial Q_0}{\partial x_j}-\cfrac{\partial Q_1}{\partial x_j}z+\cfrac{y_j}{2}Q_1 \label{brackyi}.\\
\end{eqnarray}
Then from \eqref{brackyi}, using that  $\{y_j, g\} \in  \g_{i-2}(\fn,\fs)=I_2(\cC)^{(i-1)}$ by the induction hypothesis, and in particular that $\{y_j, g\}$ depends on $\mathbf{x}$ only, we get that   $\cfrac{\partial Q_1}{\partial x_j}=0$ for every $j=0, \ldots, k$ and so $Q_1=0$. Consequently, from this and formula \eqref{brackxi} it follows that 
\begin{equation}
\label{inproli}
\cfrac{\partial Q_0}{\partial x_j}, \cfrac{\partial Q_0}{\partial y_j}\in I_2(\cC)^{(i-1)}.
\end{equation}

Using the same arguments as after relation \eqref{inprol2},  we conclude that $g=Q_0(\mathbf{x})$ belongs to $I_2(\cC)^{(i)}$. So, we proved that $\g_{i}(\fn, \fs)\subset  I_2(\cC)^{(i)}$. The opposite inclusion is obvious, so we proved \eqref{goalfin} for $i$, which concludes our proof of \eqref{goalfin} by induction and therefore, by Lemma \ref{lem1},  the proof of Theorem \ref{gns}.
\end{proof}
 
\section{End of Proof of Theorem \ref{counterthm}: the same Tanaka prolongation for different gradings}
\label{endproof}
\begin{prop}
	\label{propequal}
	For $k\geq 3$ the Tanaka prolongation $\g(\m)$ of the graded nilpotent Lie algebra $\m$ coincides (as abstract Lie algebra) with the Tanaka prolongation $\g(\fn,\fs)$ of the non-positively graded Lie algebra $\fn\oplus\fs$.  
\end{prop}
\begin{rem}
	Note that these two Lie algebras $\g(\m)$ and $\g(\fn,\fs)$ have different gradings, and the proposition states that they only coincide as abstract Lie algebras.  In fact, as we shall see in the proof, they also coincide as bi-graded Lie algebras, if we take into account both gradings. 
\end{rem}

\begin{proof}
	Consider first the Lie algebra $\g(\fn,\fs)$ and equip it with the second grading inherited from the grading of $\g'$ as specified in~\eqref{grad1} and \eqref{grad2}. This turns $\g(\fn,\fs)$ into a bi-graded Lie algebra:
	\[
	  \g(\fn,\fs) = \bigoplus_{i,j} \g_{i,j}(\fn,\fs).
	\]
	Identifying $\g(\fn,\fs)$ with the polynomials in $\mathbf{x}, \mathbf{y}, z$ as given in the previous section and in Theorem~\ref{gns} and taking into account the grading of $\g'$ given by \eqref{firstP}-\eqref{grad2} and the identification~\eqref{EFxy}, we get that $\g_{\bullet,j}(\fn,\fs)$ is spanned by homogeneous polynomials of weighted degree $j+2$, where the weight of $x_i$ is equal to $i$, the weight of $y_i$ is equal to $2-i$, and the weight of $z$ is $2$.

	Let us now prove that $\g(\fn,\fs)\subset \g(\m)$. In order to do this, it is enough to show that 
		\begin{enumerate}
		\item $\g_{\bullet,-}(\fn,\fs)=\m$; 
		\item if $u\in\g_{\bullet,j}(\fn,\fs)$, $j\ge 0$ is any homogeneous element and $[u,\m]=0$, then $u=0$.
	\end{enumerate}

Both statements immediately follow from the explicit description of $\g(\fn,\fs)$ given in Theorem~\ref{gns}. With the above 
rule of determining the second grading of $\g(\fn,\fs)$, the elements $H$, $Z_1$, and $Z_2$ from~\eqref{scoord} belong to $\g_{0,0}(\fn,\fs)$, and $Y$ belongs to $\g_{1,1}(\fn,\fs)$. Finally, all the spaces $I_{i+2}(\sigma_i\cC)$, $i\ge 0$ are spanned by maximal minors of Hankel matrices~\eqref{Hankel}, which are weighted homogeneous polynomials in $\mathbf{x}$ of weighted degree $\ge 2$, so that they belong to $\g_{r,j}(\fn,\fs)$ with $j\ge 0$. This proves property~(1) above. Property (2) follows from \eqref{EFxy}, \eqref{brackxy}, and the fact that all elements of non-negative degree correspond to nonconstant polynomials.

	\medskip
	Thus, $\g(\fn,\fs)$ is naturally embedded into $\g(\m)$. Let us prove the opposite embedding. From the second partial prolongation $\g''$ we know that $\oplus_{i\le 0}\g_i(\fn,\fs) = \fn\oplus\fs$ is naturally embedded to $\g(\m)$. Let us construct a graded subalgebra $\g^0\subset \g(\m)$ which satisfies the following three properties: 
	\begin{itemize}
		\item[(P1)] $\g(\m)=\fn\oplus \g^0$;
		\item[(P2)] for any $u\in\g^0$ such that $[u,\fn]=0$ we have $u=0$; 
		\item[(P3)] if $u\in\g^0$ satisfies $[u,\fn_{-i}]\subset \fn_{-i}$ for $i=1,2$, then $u \subset \fs$.
	\end{itemize}
	By definition of $\g(\fn, \fs)$, this would guarantee that $\g(\m)$ is embedded into $\g(\fn,\fs)$ and would complete the proof.
	
	For any integer $l$ define $\g^l(\m)=\oplus_{i\ge l}\g_i(\m)$. Let 
	\begin{equation}
	\label{V0}
	V_0 =\langle X, E_1,\dots,E_k, F_{k-1}, F_k \rangle + \g^0(\m).
	\end{equation}
	Note that the by~\eqref{grad1} the subspace~$V_0$ is a hyperplane in $\g^0(\m)$ and
    	\begin{equation}
    	\label{V0N}
    	\g^0(\m) =V_0\oplus\langle N\rangle.
    \end{equation}
	Consider an isotropy action of the subalgebra $\g^0(\m)$ on $\g(\m)/\g^0(\m)$ and let $\fh$ be the subalgebra of $\g^0(\m)$ that stabilizes the subspace $V_0$.
	
	Note that the first partial prolongation $\g'$ (see~\eqref{firstP}) is embedded into $\g(\m)$ by definition of the Tanaka prolongation. We see that
	\begin{equation}
	 \g'\cap \g^0(\m) = \langle E_0; F_0, \dots, F_{k-2} \rangle.
	\end{equation}
	As we have $[E_0,F_k] = N$ and $[F_i, E_{k-i}] = \pm N$ for $k=0,\dots,k-2$, we see from~\eqref{V0} and~\eqref{V0N} that $\g'\cap \fh = \{0\}$.
	\begin{lem}
	We have $\g'\oplus\fh = \g(\m)$.
	\end{lem} 
    \begin{proof}
    	Let $G$ be the simply connected Lie group that corresponds to the Lie algebra $\g(\m)$, and let $G^0$ be the connected subgroup in $G$ that corresponds to the subalgebra $\g^{0}(\m)$. Consider the homogeneous space $M=G/G^0$. Its tangent space $T_oM$ at the point $o=eG^0$ can be naturally identified with $\g(\m)/\g^0(\m)$. Let $H$ be the subgroup in $G^0$ that stabilizes the hyperplane $V_0/\g^0(\m)$ in $T_oM$. Then it is clear that $\fh$ is exactly the subalgebra that corresponds to this subgroup $H$. 
    	
    	Note that $V_0$ contains the subspace
    	\begin{equation}\g^{-1}(\m)/\g^0(\m) = (\m_{-1}+\g^{0}(\m))/\g^{0}(\m),
    	\end{equation}
    	which, by definition of $\g(\m)$, is stable with respect to the action of 
    	$\g^0(\m)$ and hence $G^0$. So, $H$ is essentially the stabilizer of a fixed hyperplane in $\g(\m)/\g^{-1}(\m)$ or, equivalently, a fixed  point in the projectivization $\big(\g(\m)/\g^{-1}(\m)\big)^*$. Since $\dim \, \g(\m)/\g^{-1}(\m)=k$, the codimension of $H$ in $G^0$ is at most $k$ and its codimension in $G$ is at most $2k+4$. But we have $\dim \g'=2k+4$. Since $\g'\cap \fh = \{0\}$, we see that $\g'\oplus\fh = \g(\m)$.    
    \end{proof}

	Recall that $\fn = \langle E_0,\ldots, E_{k+1},  F_0,\ldots, F_{k+1},  N \rangle$ is a subalgebra in $\g'$ and hence is contained in $\g(\m)$. Moreover $\g'= \fn \oplus \langle X \rangle$. 
	
	Hence, we see that 
	\begin{equation}\label{g0split}
	\g^0 = \fh\oplus \langle X\rangle
	\end{equation}
is indeed a subalgebra in $\g(\m)$ complementary to $\fn$, i.e. it satisfies property~(P1) above.
	
	Let us prove that $\g^0$ satisfies property~(P2).  Indeed, let $u\in \g^0$ and $[u, \fn]=0$. 
	
	First, prove that $u\in \fh$. Indeed, by \eqref{g0split} we can assume that $u=h+c X$ for some $h\in \fh$ and $c\in\mathbb R$. This implies that $0=[h+cX, E_1]=[h, E_1]+c E_2$,  i.e. $[h, E_1]=-c E_2$. Since 
	
\begin{equation}
\label{hsub}
	\fh\subset \g^0(\m),
\end{equation}	 the latter equality implies that $c=0$, so $u\in \fh$.
	
	Hence, by inclusion~\eqref{hsub} , Property (2) of Definition \ref{Tanakadef} of the Tanaka prolongation, and the fact that $\m$ is generated by $\m_{-1}$ it follows that in order to prove that $u=0$, it is sufficient to prove that $[u, \m_{-1}]=0$. Recalling that $\m_{-1}=\langle X, E_1, F_{k-1}\rangle$ and the fact that $\{E_1, F_{k-1}\} \subset \fn$,  it remains to prove that $[u, X]=0$. Assuming that it is not the case, let $u$ be a nonzero element of $\fh$ of minimal weight (with respect to the grading in $\g(\m)$) such that $[u,\fn]=0$. Since $[X, \fn]\subset \fn$ this will imply that $[[u, X], \fn]=0$. Note that by assumptions $[u, X]\neq 0$ and it has the weight one less than $u$, which contradicts the minimality of the weight of $u$ and completes the proof of property~(P2).
	
	Now  prove property~(P3). Obviously,  the space of all $\rm{ad}\, u$, $u\in \g^0$, such that $[u, \fn_{-i}]\subset \fn_{-i}$ is a subalgebra of $\mathfrak{csp}(\fn_{-1})$. Also, it contains $\rm{ad}\, X$ and from the splitting \eqref{g0split} and the fact that $\fh\in \g^0$ it follow that the space $ \langle \rm{ad}\, X \rangle $ coincides with the negative graded part of this subalgebra. This together with the maximality of the algebra $\mathfrak s$ implies property~(P3) and completes the proof of Proposition \ref{propequal}.
\end{proof}
Finally, using the explicit description of the spaces $I_{r+2}(\sigma_r\cC)$ in terms of maximal minors of Hankel matrices, we get that $\dim I_{r+2}(\sigma_r\cC)=\binom{k-r}{r+2}$, i.e. it is equal to the number  of maximal minors in matrix~\eqref{Hankel}. Further, note that $\dim \g'=2k+4$ and so 
\[
\dim \g''=2k+8+\dim I_{2}(\sigma_r\cC)=2k+8=k+6+\sum_{i=0}^2  \binom{k+2-i}{i}.
\]
From the last formula, Theorem~\ref{gns}, and Proposition \ref{propequal} it follows that   
\begin{equation}
\label{finequal1}
\dim \g(\m) = \dim\g(\fn,\fs) = k+6+\sum_{i=0}^{\lfloor k/2\rfloor+1} \binom{k+2-i}{i}   
\end{equation}
Now recall that 
\[
	\sum_{i=0}^{\lfloor k/2\rfloor+1} \binom{k+2-i}{i}=\Fib_{k+3},
\] because $\Fib_n$ can be seen as the number of ordered integer partitions (i.e. compositions) of $n-1$ by  $1$'s and $2$'s. And the number of such compositions with summands $2$ appearing exactly $i$ times is equal to $\binom{n-1-i}{i}$. This together with~\eqref{finequal1} yields
\begin{equation}
\label{finequal2}
\dim \g(\m) = \dim\g(\fn,\fs) = k+6+\Fib_{k+3}.   
\end{equation}
and hence identity \eqref{mainid1}.

Now let us prove the identity \eqref{mainid2}. As in the proof of Proposition~\ref{propequal}, identify $\g_i(\m)$ with the space of homogeneous polynomials in $\mathbf{x}$,  $\mathbf{y}$, $z$ of weighted degree $i+2$, where the weight of the variable $x_i$ is equal to $i$, the weight of the variable $y_i$ is equal to $2-i$, and the weight of variable of $z$ is equal to $2$.  Note that with this rule all elements of $\g_i(\m)$ with $i\geq 2$ are weighted homogeneous polynomials in $x_0,x_1,\dots,x_k$ of weighted degree $i+2$. They are spanned by maximal minors of Hankel matrices \eqref{Hankel}. As we will see in the next paragraph, for $k\geq 3$ the maximal weighted degree of such minors is greater or equal to $4$. Hence, the maximal positive nonzero $\nu$ for which $\g_\nu(\m)\neq 0$ is equal to this maximal weighted degree minus $2$. 

Consider the case of even and odd $k$ separately. If $k\geq 4$ is even, then the minor with maximal weighted degree corresponds to the determinant of the square Hankel matrix \eqref{Hankel} with $r=\cfrac{k}{2}-1$ and it has the weighted degree equal to $2+4+\ldots k=\cfrac{(k+2)k}{4}=\left\lfloor\cfrac{(k+1)^2}{4}\right\rfloor\ge 6$, which implies \eqref{mainid2} for this case. If $k\geq 3$ is odd, then the minor with maximal weighted degree corresponds to the minor of the Hankel matrix \eqref{Hankel} with $r=\cfrac{k-1}{2}-1$, obtained by removing the first column. It has weighted degree $1+3+\ldots +k=\cfrac{(k+1)^2}{4}\ge 4$, which also implies \eqref{mainid2} for this case.

This completes the proof of Theorem~\ref{counterthm}.

\section{Geometry behind our models}
\label{geomsec}

In this section we briefly describe the original geometric constructions that led to our models (for more detail, see \cite{DZ2016}). 
The key point is that, with any bracket generating rank~$3$ distribution $D$ on a $(k+4)$-dimensional manifold $M$ with $\dim D^2=D+[D, D]$ of rank~$6$ one can naturally associate another geometric structure consisting of a contact distribution $\Delta$ on a $(2k+3)$-dimensional manifold $\Sigma$ and a fixed curve of symplectic flags in each $\Delta_{\lambda}$, $\lambda\in \Sigma$. The algebras of infinitesimal symmetries of the original rank $3$ distribution and the new structure are isomorphic. 

In the case of the flat rank $3$ distribution $D_\m$ with the Tanaka symbol $\m$ as in \eqref{grad1} this curve of the symplectic flags on each fiber of $\Delta$ coincides, up to a symplectic transformation, with the curve $\mathcal F_X$ defined in Remark \ref{flatcurve}. Taking into account that the Tanaka prolongation $\g(\m)$ is isomorphic to the algebra of infinitesimal symmetries of $D_\m$, the Heisenberg algebra $\fn$ is isomorphic to the Tanaka symbol of the contact distribution $\Delta$, and for $k\geq 3$ the algebra of infinitesimal symmetries of $\mathcal F_X$ is isomorphic to  $\fs$, we obtain the geometric justification of Proposition \ref{propequal}.

In more details, let, as before, $D^{1}:=D$  Assume that $D^2$ is a  distribution as well. Denote by $(D^j)^{\perp}\subset T^*M$, $j=1,2$  the annihilator of $D^j$, namely 
\begin{equation}
(D^j)^{\perp}=\{(p,q)\in T_q^*M\mid p\cdot v=0\quad\forall\,v\in D^j(q)\}.
\end{equation}
Let $\pi\colon T^*M\mapsto M$ be the canonical projection. For any $\lambda\in T^*M$, $\lambda=(p,q)$, $q\in M$, $p\in T_q^*M$, let
$\varsigma(\lambda)(\cdot)=p(\pi_*\cdot)$ be the tautological Liouville form
and $\hat\sigma=d\varsigma$ be the standard symplectic structure on $T^*M$.

Since $\dim D=3$, the submanifold $D^\perp$ has odd codimension in $T^*M$, and the kernels of the restriction $\hat\sigma|_{D^\perp}$ are non-trivial. Moreover, it is easy to see (\cite{DZ2016}) that on
$\widehat\Sigma=D^\perp\backslash (D^2)^\perp$ these kernels are one-dimensional. They form a \emph{characteristic line distribution} on $\widehat\Sigma$, which will be denoted by $\widehat{\X}$. This line distribution defines in turn a \emph{characteristic 1-foliation} $\widehat \A$ on $D^\perp\backslash(D^2)^\perp$.

Now we define the same objects on the
projectivization of $D^\perp\backslash(D^2)^\perp$. As homotheties of the
fibers of $D^\perp$ preserve the characteristic line distribution, the
projectivization induces the \emph{characteristic line distribution} $\X$  on $\Sigma = \PDD$. Similar to above, it defines \emph{the characteristic $1$-foliation} $\A$ on $\Sigma$, and its leaves are called the
\emph{characteristic curves of the distribution $D$}.
 
The distribution $\X$ can be defined equivalently in the following way. Take the corank $1$ distribution on $\widehat\Sigma$ given by the Pfaffian equation $\varsigma|_{D^\perp}=0$ and push it forward under
projectivization to $\Sigma$. In this way we obtain a corank 1 distribution on $\Sigma$, which will be denoted by $\widetilde\Delta$. The distribution $\widetilde\Delta$ defines a quasi-contact structure on the even dimensional manifold $\Sigma$ and $\X$ is exactly the Cauchy characteristic distribution of this quasi-contact structure, i.e. the maximal subdistribution of $\widetilde \Delta$ such that
\begin{equation}
\label{Cauchy}
[\X, \widetilde \Delta]\subset \widetilde \Delta.
\end{equation}

Fix a point $\lambda_0\in \Sigma$. Then locally the space of leaves of $\A$ is a well-defines smooth manifold that will be called the \emph{characteristic leaf space} and will be denote by $\N$. Let  $\Phi\colon\Sigma\to\N$ be the canonical projection to the quotient manifold defined in a neighborhood of $\lambda_0$. By~\eqref{Cauchy} the push-forward of $\widetilde \Delta$ by $\Phi$ defines the contact distribution on $\N$, which will be denoted by $\Delta$.  

In this way we obtain the following double fibration on $\Sigma$:
\begin{equation}
\label{doublef}
\xymatrix{
	& {\Sigma} \ar[ld]_{\pi} \ar[rd]^{\Phi} &\\
	 M & &  \mathcal N} 
\end{equation}

The double fibration \eqref{doublef} induces the additional structure on each fiber of the contact distribution $\Delta$ of the characteristic leaf space $\N$. Namely, let $\mJ^0$ be the distribution of tangent spaces to the fibers of the fibration $\pi\colon\Sigma\rightarrow M$ and for every $\gamma\in\N$ define a curve 
\begin{equation}
\label{jac}
\lambda\mapsto J_\gamma^{0}(\lambda)=\Phi_*(\mJ^0(\lambda)), \quad \forall\lambda\in \gamma.
\end{equation}
In \eqref{jac} $\gamma$ is considered as the leaf of the characteristic foliation $\A$. It is easy to see (\cite{DZ2016}) that $\{\lambda\mapsto J_\gamma^{0}(\lambda):\lambda\in \gamma\}$ is a curve of isotropic subspaces of $\Delta(\gamma)$ with respect to the canonical conformal symplectic structure. Using operations of osculation and skew-orthogonal complement, we can build the following symplectic flag of subspaces of $\Delta(\gamma)$ from the curve~\eqref{jac}:
\begin{equation}
\label{jacflag}
\lambda \mapsto \left\{ 0\subset \dots \subset J_\gamma^1(\lambda)\subset J_\gamma^{0}(\lambda)\subset J_\gamma^{-1}(\lambda)
\subset \dots \subset \Delta(\gamma)\right
\},\quad\lambda\in\gamma,
\end{equation}
where the curve $\{\lambda \mapsto J_\gamma^{-i}(\lambda), \lambda\in \gamma\}$  is the $i$-th  osculation of the curve \eqref{jac} for $j\geq 0$, and $J_\gamma^i(\lambda)$ is th skew-orthogonal complement of $J^{-i-1}(\lambda)$ with respect to the canonical conformal symplectic structure on $\Delta(\gamma)$. So the contact distribution $\Delta$ on $\N$ together with the curve of symplectic flags~\eqref{jacflag}  on each fiber $\Delta$ is exactly the structure we mentioned in the first paragraph of this section.

Finally, if we implement the construction of this section to rank $3$ distribution $D_\m$ with constant symbol $\m$ as in \eqref{grad1}, then it is easy to see that 
\begin{enumerate}
\item $\Sigma$ can be identified with the simply connected Lie group $G'$  with the Lie algebra isomorphic to $\g'$  defined in~\eqref{firstP};
\item the characteristic line distribution is generated by the left invariant vector on $G'$ corresponding to the element $X\in\g'$;
\item the leaf space $\N$ can be identified with the quotient of $G'$ by the subgroup $\exp(\ad tX)$. This quotient is isomorphic to the Heisenberg Lie group with the Lie algebra $\fn$;
\item the contact distribution $\Delta$ is the left invariant distribution on $\N$ which is equal to $\fn_{-1}$ at the identity;
\item the curve of symplectic flags \eqref{jacflag} at the identity of the group $\N$ coincides with the curve $\mathcal F_X$  described in Remark~\ref{flatcurve}, and the curve \eqref{jacflag} at any other point of $\mathcal N$ is obtained by left shifts in~$\N$. 
\end{enumerate}


\begin{thebibliography}{99}
	\bibitem{bryant} R.~Bryant, \emph{Conformal geometry and 3-plane fields on
		6-manifolds}, Proceedings of the RIMS symposium ``Developments of
	Cartan geometry and related mathematical problems'' (24-27 October
	2005).
	\bibitem{CapSlovak}
	A. $\rm{\check{C}}$ap and J. Slov${\rm \acute{a}}$k, \emph{Parabolic Geometries I:  Background and general theory}, Mathematical
	Surveys and Monographs, vol. 154, American
	Mathematical Society, Providence, RI, 2009
	\bibitem{DZ2016} B.~Doubrov, I.~Zelenko, \emph{On local geometry of vector distributions with given Jacobi symbols}, arXiv:1610.09577.
	\bibitem{DZ2013} B.~Doubrov, I.~Zelenko, \emph{Geometry of curves in generalized flag varieties}, Transform.\ Groups, vol.\,18, No.\,2 (2013), 361--383.
	\bibitem{FH1991} W.~Fulton, J.~Harris, \emph{Representation theory: a first course}, Graduate Texts in Mathematics, vol.~129, Springer, 1991. 
	\bibitem{IK1999} A.~Iarrobino, V.~Kanev, \emph{Power sums, Gorenstein algebras, and determinantal loci.} Lecture Notes in Mathematics, vol. 1721. Springer, Berlin, 1999.
	\bibitem{kobnag} S.~Kobayashi, T.~Nagano, \emph{On filtered Lie algebras and geometric structures III},
	J.~Math.\ Mech., v.~14, 1965, pp.~679--706.
	\bibitem{kob} S.~Kobayashi, \emph{Transformation Groups in Differential Geometry}, Classics in Mathematics, Springer, 1995, 182 pages.
	\bibitem{LO2013} J.M.~Landsberg, Giorgio Ottaviani, \emph{Equations for secant varieties of Veronese and other varieties}, Ann. Mat. Pura Appl. (4) {\bf 192} (2013), no. 4, 569--606.
	\bibitem{stern} S.~Sternberg, \emph{Lectures on differential geometry},
	Prentice Hall, Englewood Cliffs, N.J., 1964.
  \bibitem{tan} N.~Tanaka, \emph{On differential systems, graded Lie
		algebras and pseudo-groups}, J.\ Math.\ Kyoto.\ Univ., \textbf{10}
	(1970), 1--82.
	\bibitem{yam} K. Yamaguchi, \emph{Differential Systems Associated with Simple Graded Lie algebras},  Advanced Studies in Pure Mathematics \textbf{22}, 1993, Progress in Differential Geometry, 413--494.
\end{thebibliography}
\end{document}